\documentclass{amsart}%
\usepackage{amsfonts}
\usepackage{amsmath}
\usepackage{amssymb}
\usepackage{graphicx}%
\providecommand{\U}[1]{\protect\rule{.1in}{.1in}}
\newtheorem{theorem}{Theorem}
\theoremstyle{plain}

\newtheorem{corollary}{Corollary}

\newtheorem{definition}{Definition}
\newtheorem{example}{Example}

\newtheorem{lemma}{Lemma}

\newtheorem{proposition}{Proposition}
\newtheorem{remark}{Remark}

\numberwithin{equation}{section}
\begin{document}
\title[Quasi $J$-submodules]{Quasi $J$-submodules}
\author{Ece YETKIN\ CELIKEL}
\address{Department of Electrical-Electronics Engineering, Faculty of Engineering,
Hasan Kalyoncu University, Gaziantep, Turkey.}
\email{ece.celikel@hku.edu.tr, yetkinece@gmail.com}
\author{Hani A. Khashan }
\address{Department of Mathematics, Faculty of Science, Al al-Bayt University,Al
Mafraq, Jordan.}
\email{hakhashan@aabu.edu.jo}
\thanks{This paper is in final form and no version of it will be submitted for
publication elsewhere.}
\date{January, 2021}

\begin{abstract}
Let $R$ be a commutative ring with identity and $M$ be a unitary $R$-module.
The aim of this paper is to extend the notion of quasi $J$-ideals of
commutative rings to quasi $J$-submodules of modules. We call a proper
submodule $N$ of $M$ a quasi $J$-submodule if whenever $r\in R$ and $m\in M$
such that $rm\in N$ and $r\notin(J(R)M:M)$, then $m\in M$-$rad(N)$. We present
various properties and characterizations of this concept (especially in
finitely generated faithful multiplication modules). Furthermore, we provide
new classes of modules generalizing presimplifiable modules and justify their
relation with (quasi) $J$-submodules. Finally, for a submodule $N$ of $M$ and
an ideal $I$ of $R$, we characterize the quasi $J$-ideals of the idealization
ring $R(+)M$.

\end{abstract}
\maketitle

\section{Introduction}

All rings considered in this paper are commutative with identity elements, and
all modules are unital. Let $R$ be a ring and $N$ be a submodule of an
$R$-module $M.$ By $Z(R)$, $reg(R),$ $N(R)$, $J(R)$, $Z(M)$ and $M$-$rad(N),$
we denote the set of zero-divisors of $R$, the set of regular elements in $R$,
the nil radical of $R$, the Jacobian radical of $R,$ the set of all zero
divisors on $M$; i.e. $\{r\in R:rm=0$ for some $0\neq m\in M\}$ and the
intersection of all prime submodules of $M$ containing $N$, respectively. For
submodules $N$ of $M$, the residual of $N$ by $M$, $(N:M)$ denotes the ideal
$\{r\in R:rM\subseteq N\}$. In particular, the ideal $(0:M)$ is called the
annihilator of $M$. Moreover, if $I$ is an ideal of $R$, then the residual
submodule $N$ by $I$ is $[N:_{M}I]=\{m\in M:Im\subseteq N\}$. An $R$-module
$M$ is a multiplication module if every submodule $N$ of $M$ has the form $IM$
for some ideal $I$ of $R$. Equivalently, $N=(N:M)M$, \cite{Bernard}.

In 2015, Mohamadian \cite{Moh} introduced the concept of $r$-ideals in
commutative rings. A proper ideal $I$ of a ring $R$ is called an $r$-ideal if
whenever $a,b\in R$, $ab\in I$ and $Ann\left(  a\right)  =0$ imply that $b\in
I$ where $Ann\left(  a\right)  =\left\{  r\in R:ra=0\right\}  $. As a subclass
of $r$-ideals, Tekir et. al. \cite{Tekir} defined a proper ideal $I$ of $R$ to
be $n$-ideal if whenever $a,b\in R$ such that $ab\in I$ and $a\notin N(R)$,
then $b\in I$. Later, Khashan and Bani-Ata \cite{Hani} generalized this notion
to $J$-ideals and $J$-submodules. A proper ideal $I$ of $R$ is said to be a
$J$-ideal if whenever $a,b\in R$ such that $ab\in I$ and $a\notin J(R)$, then
$b\in I$. A proper submodule $N$ of an $R$-module $M$ is said to be a
$J$-submodule if whenever $r\in R$ and $m\in M$ with $rm\in N$ and
$r\notin(J(R)M:M)$, then $m\in N$. Generalizing the idea of $J$-ideals, as a
very recent study \cite{Haniece}, the class of quasi $J$-ideals has been
defined and studied. A proper ideal $I$ of $R$ is said to be a quasi $J$-ideal
if $\sqrt{I}=\left\{  x\in R:x^{n}\in I\text{ for some }n\in%
%TCIMACRO{\U{2124} }%
%BeginExpansion
\mathbb{Z}
%EndExpansion
\right\}  $ is a $J$-ideal.

The purpose of the present work is to generalize the notions of quasi
$J$-ideals and $J$-submodules by defining and studying quasi $J$-submodules.
We call a proper submodule $N$ of $M$ a quasi $J$-submodule if whenever $r\in
R$ and $m\in M$ such that $rm\in N$ and $r\notin(J(R)M:M)$, then $m\in
M$-$rad(N)$. In Section 2, we investigate many general properties of quasi
$J$-submodules of an $R$-module $M$ with various examples. $J$-submodules are
obviously a quasi $J$-submodules, but the converse of this implication is not
true in general (Example \ref{e1}). Among many other results in this section,
we investigate quasi $J$-submodules under various contexts of constructions
such as homomorphic images, direct products and localizations (see
Propositions \ref{8}, \ref{d} and \ref{l}). In 1997, the notion of
presimplifiable modules was first studied by Anderson and and Valdes-Leon
\cite{Anderson3} as $R$-modules with property that $Z(M)\subseteq J(R)$.
Motivated from this concept, we introduce new generalizations of
presimplifiable modules which are quasi presimplifiable, $J$-presimplifiable
and quasi $J$-presimplifiable modules. Example \ref{ep} is given to show that
these generalizations are proper. The main result of this section (Theorem
\ref{tp}) gives a relation between quasi $J$-submodules (resp. $J$-submodules)
and quasi presimplifiable (resp. $J$-presimplifiable) modules which enables us
to construct more examples for quasi $J$-submodules. Precisely, a submodule
$N$ of an $R$-module $M$ contained in $J(R)M$ is a quasi $J$-submodule (resp.
$J$-submodule) if and only if the quotient $M/N$ is a non-zero quasi
$J$-presimplifiable (resp. $J$-presimplifiable) $R$-module.

The last section deals with quasi $J$-submodules in multiplication modules. In
Theorem \ref{1}, Theorem \ref{3} and Proposition \ref{7}, we present many
properties and characterizations for quasi $J$-submodules of multiplication
modules (especially, in finitely generated faithful multiplication modules).
In particular, in such a module $M$, we characterize quasi $J$-submodules $N$
as those in which the residual ideal $(N:M)$ is a quasi $J$-ideal. It is shown
in Theorem \ref{max} that every quasi $J$-submodule of an $R$-module $M$ is
contained in a maximal quasi $J$-submodule of $M$. Furthermore, if $M$ is
finitely generated faithful multiplication, then a maximal quasi $J$-submodule
of $M$ is a $J$-submodule. For a submodule $N$ of $M$ and an ideal $I$ of $R$,
we finally (Theorem \ref{6}) give a characterization of $J$-ideals in the
idealization ring $R(+)M$.

\section{General properties of Quasi $J$-submodules}

In this section, among other results concerning the general properties of
quasi $J$-submodules, some characterizations of this notion will be
investigated. Moreover, the relations among quasi $J$-submodules and some
other types of submodules will be clarified.

First, we present the fundamental definition of quasi $J$-submodules which
will be studied in this paper.

\begin{definition}
Let $R$ be a ring and let $M$ be an $R$-module. A proper submodule $N$ of $M$
is called a quasi $J$-submodule if whenever $r\in R$ and $m\in M$ such that
$rm\in N$ and $r\notin(J(R)M:M)$, then $m\in M$-$rad(N)$.
\end{definition}

It is clear that any $J$-submodule of $M$ is a quasi $J$-submodule. However,
in the next example we can see that the converse is not true in general.

\begin{example}
\label{e1}Consider the $%
%TCIMACRO{\U{2124} }%
%BeginExpansion
\mathbb{Z}
%EndExpansion
$-module $M=%
%TCIMACRO{\U{2124} }%
%BeginExpansion
\mathbb{Z}
%EndExpansion
_{4}$. Then one can directly see that $M$-$rad(\left\langle \bar
{0}\right\rangle )=\left\langle \bar{2}\right\rangle $. Now, let $r\in%
%TCIMACRO{\U{2124} }%
%BeginExpansion
\mathbb{Z}
%EndExpansion
$ and $\bar{k}\in%
%TCIMACRO{\U{2124} }%
%BeginExpansion
\mathbb{Z}
%EndExpansion
_{4}$ such that $r\cdot\bar{k}\in\left\langle \bar{0}\right\rangle $ and
$r\notin(J(%
%TCIMACRO{\U{2124} }%
%BeginExpansion
\mathbb{Z}
%EndExpansion
)%
%TCIMACRO{\U{2124} }%
%BeginExpansion
\mathbb{Z}
%EndExpansion
_{4}:%
%TCIMACRO{\U{2124} }%
%BeginExpansion
\mathbb{Z}
%EndExpansion
_{4})=(0:%
%TCIMACRO{\U{2124} }%
%BeginExpansion
\mathbb{Z}
%EndExpansion
_{4})=\left\langle 4\right\rangle $. Then clearly $\bar{k}\in\left\langle
\bar{2}\right\rangle =M-rad(\left\langle \bar{0}\right\rangle )$ and so
$\left\langle \bar{0}\right\rangle $ is a quasi $J$-submodule. On the other
hand, $\left\langle \bar{0}\right\rangle $ is not a $J$-submodule since for
example, $2\cdot\bar{2}=\bar{0}$ but $2\notin(0:%
%TCIMACRO{\U{2124} }%
%BeginExpansion
\mathbb{Z}
%EndExpansion
_{4})$ and $\bar{2}\neq\bar{0}$.
\end{example}

Following \cite{Suat}, a proper submodule $N$ of an $R$-module $M$ is called
an $r$-submodule (resp. $sr$-submodule) if whenever $am\in N$ with
$ann_{M}(a)=0_{M}$ (resp. $ann_{R}(m)=0_{R}$), then $m\in N$ (resp.
$a\in(N:M)$) for each $a\in R$ and $m\in M$.

In general, the class of (quasi) $J$-submodules is not comparable with the
classes of $r$-submodules, $sr$-submodules and prime submodules.

\begin{example}
(1) The submodule $N=\left\langle \bar{0}\right\rangle $ is a quasi
$J$-submodule of the $%
%TCIMACRO{\U{2124} }%
%BeginExpansion
\mathbb{Z}
%EndExpansion
$-module $%
%TCIMACRO{\U{2124} }%
%BeginExpansion
\mathbb{Z}
%EndExpansion
_{4}$ which is not an $r$-submodule, an $sr$-submodule or a prime submodule.

(2) The submodule $\left\langle 2\right\rangle $ is a prime submodule of the $%
%TCIMACRO{\U{2124} }%
%BeginExpansion
\mathbb{Z}
%EndExpansion
$-module $%
%TCIMACRO{\U{2124} }%
%BeginExpansion
\mathbb{Z}
%EndExpansion
$ which is not a (quasi) $J$-submodule.

(3) The submodule $N=\left\langle \bar{2}\right\rangle $ is an $r$-submodule
and $sr$-submodule of the $%
%TCIMACRO{\U{2124} }%
%BeginExpansion
\mathbb{Z}
%EndExpansion
$-module $%
%TCIMACRO{\U{2124} }%
%BeginExpansion
\mathbb{Z}
%EndExpansion
_{6}$, \cite[Example 1]{Suat} which is not a (quasi) $J$-submodule.
\end{example}

In the following result, we give a characterization for quasi $J$-submodules
of an $R$-module $M$.

\begin{proposition}
\label{eq1}Let $M$ be an $R$-module and $N$ be a proper submodule of $M$. The
following are equivalent:

\begin{enumerate}
\item $N$ is a quasi $J$-submodule of $M$.

\item If $r\in R-(J(R)M:M)$ and $K$ is a submodule of $M$ with $rK\subseteq
N$, then $K\subseteq M$-$rad(N)$.

\item If $A\nsubseteq(J(R)M:M)$ is an ideal of $R$ and $K$ is a submodule of
$M$ with $AK\subseteq N$, then $K\subseteq M$-$rad(N)$.
\end{enumerate}
\end{proposition}

\begin{proof}
(1)$\Rightarrow$(2) Assume $N$ is a quasi $J$-submodule. Suppose $r\in
R-(J(R)M:M)$ and $K$ is a submodule of $M$ with $rK\subseteq N$. Then for all
$k\in K$, $rk\in N$ and so $k\in M$-$rad(N)$. Therefore, $K\subseteq
M$-$rad(N)$ as needed.

(2)$\Rightarrow$(3) Suppose $AK\subseteq N$ for an ideal $A\nsubseteq
(J(R)M:M)$ and a submodule $K$ of $M$. Then for $r\in A-(J(R)M:M)$, we have
$rK\subseteq N$ and so by assumption, $K\subseteq M$-$rad(N)$.

(3)$\Rightarrow$(1) Let $r\in R$ and $m\in M$ such that $rm\in N$ and
$r\notin(J(R)M:M)$. Then $\left\langle r\right\rangle \left\langle
m\right\rangle \subseteq N$ and $\left\langle r\right\rangle \nsubseteq
(J(R)M:M)$ and so $m\in\left\langle m\right\rangle \subseteq M$-$rad(N)$ and
we are done.
\end{proof}

\begin{lemma}
\cite{Lu}\label{4}Let $\varphi:M_{1}\longrightarrow M_{2}$ be an $R$-module epimorphism.Then

\begin{enumerate}
\item If $N$ is a submodule of $M_{1}$ and $\ker\left(  \varphi\right)
\subseteq N$, then $\varphi\left(  M_{1}\text{-}rad(N)\right)  =M_{2}%
$-$rad(\varphi\left(  N\right)  )$.

\item If $K$ is a submodule of $M_{2}$, then $\varphi^{-1}\left(
M_{2}\text{-}rad(K)\right)  =M_{1}$-$rad(\varphi^{-1}\left(  K\right)  )$.
\end{enumerate}
\end{lemma}

\begin{proposition}
\label{8}Let $\varphi:M_{1}\longrightarrow M_{2}$ be an $R$-module
epimorphism. Then

\begin{enumerate}
\item If $N$ is a quasi J-submodule of $M_{1}$ with $\ker\left(
\varphi\right)  \subseteq N$, then $\varphi\left(  N\right)  $ is a quasi
J-submodule of $M_{2}$.

\item If $K$ is a quasi J-submodule of $M_{2}$ with $\ker\left(
\varphi\right)  \subseteq J\left(  R\right)  M_{1}$, then $\varphi^{-1}\left(
K\right)  $ is a quasi J-submodule of $M_{1}$.
\end{enumerate}
\end{proposition}

\begin{proof}
$\left(  1\right)  $ Suppose $\varphi\left(  N\right)  =M_{2}=\varphi\left(
M_{1}\right)  $ and let $m_{1}\in M_{1}$. Then $\varphi\left(  m_{1}\right)
=\varphi\left(  n\right)  $ for some $n\in N$ and so $m_{1}-n\in\ker\left(
\varphi\right)  \subseteq N$. So, $m_{1}\in N$ and $N=M_{1}$ which is a
contradiction. Hence, $\varphi\left(  N\right)  $ is proper in $M_{2}$. Let
$r\in R$ and $m_{2}\in M_{2}$ such that $rm_{2}\in\varphi\left(  N\right)  $
and $r\notin\left(  J\left(  R\right)  M_{2}:M_{2}\right)  $. Choose $m_{1}\in
M_{1}$ such that $\varphi\left(  m_{1}\right)  =m_{2}$. Then $rm_{2}%
=r\varphi\left(  m_{1}\right)  =\varphi\left(  rm_{1}\right)  \in
\varphi\left(  N\right)  $. Thus, $\varphi\left(  rm_{1}-a\right)  =0$ for
some $a\in N$ and so $rm_{1}-a\in\ker\left(  \varphi\right)  \subseteq N$. It
follows that $rm_{1}\in N$. Moreover, we have $r\notin\left(  J\left(
R\right)  M_{1}:M_{1}\right)  $. Indeed, if $rM_{1}\subseteq J\left(
R\right)  M_{1}$, then $rM_{2}=r\varphi\left(  M_{1}\right)  =\varphi
(rM_{1})\subseteq\varphi\left(  J\left(  R\right)  M_{1}\right)  =J\left(
R\right)  \varphi\left(  M_{1}\right)  =J\left(  R\right)  M_{2}$ which is a
contradiction. Since $N$ is a quasi J-submodule, then $m_{1}\in M_{1}%
$-$rad(N)$. Thus, $m_{2}=\varphi\left(  m_{1}\right)  \in\varphi\left(
M_{1}\text{-}rad(N)\right)  =M_{2}$-$rad(\varphi\left(  N\right)  )$ by Lemma
\ref{4} and $\varphi\left(  N\right)  $ is a quasi J-submodule of $M_{2}$.

$\left(  2\right)  $ Clearly, $\varphi^{-1}\left(  K\right)  $ is proper in
$M_{1}$. Let $r\in R$ and $m_{1}\in M_{1}$ such that $rm_{1}\in\varphi
^{-1}\left(  K\right)  $ and $r\notin\left(  J\left(  R\right)  M_{1}%
:M_{1}\right)  $. Then $r\varphi\left(  m_{1}\right)  =\varphi\left(
rm_{1}\right)  \in K$. We prove that $r\notin\left(  J\left(  R\right)
M_{2}:M_{2}\right)  $. Suppose on the contrary that $rM_{2}\subseteq J\left(
R\right)  M_{2}$. Then $r\varphi\left(  M_{1}\right)  \subseteq J\left(
R\right)  \varphi\left(  M_{1}\right)  $ and so $\varphi\left(  rM_{1}\right)
\subseteq\varphi\left(  J\left(  R\right)  M_{1}\right)  $. Now, if $x\in
rM_{1}$, then $\varphi\left(  x\right)  \in\varphi\left(  rM_{1}\right)
\subseteq\varphi\left(  J\left(  R\right)  M_{1}\right)  $ and hence
$x-t\in\ker\left(  \varphi\right)  \subseteq J\left(  R\right)  M_{1}$\ for
some $t\in J\left(  R\right)  M_{1}$. It follows that $x\in J\left(  R\right)
M_{1}$ and $rM_{1}\subseteq J\left(  R\right)  M_{1}$ which is a
contradiction. Since $K$ is a quasi $J$-submodule of $M_{2}$, then
$\varphi\left(  m_{1}\right)  \in M_{2}$-$rad(K)$. Therefore, $m_{1}\in
\varphi^{-1}\left(  M_{2}\text{-}rad(K)\right)  =M_{1}$-$rad(\varphi
^{-1}\left(  K\right)  )$ by Lemma \ref{4} and the result follows.
\end{proof}

\begin{corollary}
Let $N$ and $L$ be submodules of an $R$-module $M$ with $L\subseteq N$. If $N$
is a quasi $J$-submodule of $M$, then $N/L$ is a quasi $J$-submodule of $M/L$.
\end{corollary}

\begin{proposition}
\label{d}Let $M_{1},M_{2},...,M_{k}$ be $R$-modules and consider the
$R$-module $M=M_{1}\times M_{2}\times\cdots\times M_{k}$.

(1) If $N=N_{1}\times N_{2}\times\cdots\times N_{k}$ is a quasi $J$-submodule
of $M$, then $N_{i}$ is a quasi $J$-submodule of $M_{i}$ for all $i$ such that
$N_{i}\neq M_{i}$.

(2) If $N_{j}$ is a quasi $J$-submodule of $M_{j}$ for some $j\in\left\{
1,2,...,k\right\}  $, then $N=M_{1}\times M_{2}\times\cdots\times N_{j}%
\times\cdots\times M_{k}$ is a quasi $J$-submodule of $M$.
\end{proposition}

\begin{proof}
(1) With no loss of generality, we assume $N_{1}\neq M_{1}$ and prove that
$N_{1}$ is a quasi $J$-submodule of $M_{1}$. Let $r\in R$ and $m\in M_{1}$
such that $rm\in N_{1}$ and $r\notin(J(R)M_{1}:M_{1})$. Then $r.(m,0,...,0)\in
N$ and clearly $r\notin(J(R)M:M)$. It follows that $(m,0,...,0)\in M$-$rad(N)$
and so $m\in M$-$rad(N_{1})$ as required.

(2) With no loss of generality, suppose $N_{1}$ is a quasi $J$-submodule of
$M_{1}$. Let $r\in R$ and $(m_{1},m_{2},...,m_{k})\in M_{1}\times M_{2}%
\times\cdots\times M_{k}$ such that $(r.m_{1},r.m_{2},...,r.m_{k}%
)=r.(m_{1},m_{2},...,m_{k})\in N_{1}\times M_{2}\times\cdots\times M_{k}$ and
$r\notin(J(R)M:M)$. Then $rm_{1}\in N_{1}$ and clearly $r\notin(J(R)M_{1}%
:M_{1})$. Therefore, $m_{1}\in M_{1}$-$rad(N_{1})$ and then $(m_{1}%
,m_{2},...,m_{k})\in M$-$rad(N_{1}\times M_{2}\times\cdots\times M_{k})$.
\end{proof}

\begin{remark}
(1) If $N_{1}$ and $N_{2}$ are quasi $J$-submodules of $R$-modules $M_{1}$ and
$M_{2}$ respectively, then $N_{1}\times N_{2}$ need not be a quasi
$J$-submodule of $M_{1}\times M_{2}$. For example $\bar{0}$ and $0$ are quasi
$J$-submodules of the $%
%TCIMACRO{\U{2124} }%
%BeginExpansion
\mathbb{Z}
%EndExpansion
$-modules $%
%TCIMACRO{\U{2124} }%
%BeginExpansion
\mathbb{Z}
%EndExpansion
_{4}$ and $%
%TCIMACRO{\U{2124} }%
%BeginExpansion
\mathbb{Z}
%EndExpansion
$ respectively. However, $\bar{0}\times0$ is not a quasi $J$-submodule of $%
%TCIMACRO{\U{2124} }%
%BeginExpansion
\mathbb{Z}
%EndExpansion
_{4}\times%
%TCIMACRO{\U{2124} }%
%BeginExpansion
\mathbb{Z}
%EndExpansion
$ as $4.(\bar{1},0)\in$ $\bar{0}\times0$ but $4\notin(\bar{0}\times0:%
%TCIMACRO{\U{2124} }%
%BeginExpansion
\mathbb{Z}
%EndExpansion
_{4}\times%
%TCIMACRO{\U{2124} }%
%BeginExpansion
\mathbb{Z}
%EndExpansion
)$ and $(\bar{1},0)\notin M$-$rad(\bar{0}\times0)=2%
%TCIMACRO{\U{2124} }%
%BeginExpansion
\mathbb{Z}
%EndExpansion
_{4}\times0$.

(2) The condition $\ker\left(  \varphi\right)  \subseteq N$ in (1) of
Proposition \ref{8} is not necessary. Indeed, let $M_{1}$ and $M_{2}$ be
$R$-modules and $\varphi:M_{1}\times M_{2}\rightarrow M_{1}$ be the projection
epimorphism. If $N_{1}$ and $N_{2}$ are proper submodules of $M_{1}$ and
$M_{2}$ and $N_{1}\times N_{2}$ is a quasi $J$-submodule of $M_{1}\times
M_{2}$, then $\varphi(N_{1}\times N_{2})=N_{1}$ is a quasi $J$-submodule of
$M_{1}$. However, $\ker\left(  \varphi\right)  =0\times M_{2}\nsubseteq
N_{1}\times N_{2}$.
\end{remark}

Let $I$ be a proper ideal of $R$ and $N$ be a submodule of an $R$-module $M$ .
In the following proposition, the notations $Z_{I}(R)$ and $Z_{N}(M)$ denote
the sets $\{r\in R:rs\in I$ for some $s\in R\backslash I\}$ and $\{r\in
R:rm\in N$ for some $m\in M\backslash N\}$.

\begin{proposition}
\label{l}Let $S$ be a multiplicatively closed subset of a ring $R$ such that
$S^{-1}(J(R))=J(S^{-1}R)$ and $M$ be an $R$-module. Then

\begin{enumerate}
\item If $N$ is a quasi $J$-submodule of $M$ and $S^{-1}N\neq S^{-1}M$, then
$S^{-1}N$ is a quasi $J$-submodule of the $S^{-1}R$-module $S^{-1}M$.

\item If $S^{-1}N$ is a quasi $J$-submodule of $S^{-1}M$ and $S\cap
Z_{(J(R)M:M)}(R)=S\cap Z_{M\text{-}rad(N)}(M)=\emptyset$, then $N$ is a quasi
$J$-submodule of $M$.
\end{enumerate}
\end{proposition}

\begin{proof}
(1) Suppose that $\frac{r}{s_{1}}\frac{m}{s_{2}}\in S^{-1}N$. Then $urm\in N$
for some $u\in S.$ Since $N$ is a quasi $J$-submodule, then either
$ur\in(J(R)M:M)$ or $m\in M$-$rad(N)$. If $ur\in(J(R)M:M)$, then $\frac
{r}{s_{1}}=\frac{ur}{us_{1}}\in S^{-1}(J(R)M:M)=(S^{-1}J(R)$ $S^{-1}%
M:S^{-1}M)=(J(S^{-1}R)$ $S^{-1}M:S^{-1}M).$ If $m\in M$-$rad(N)$, then
$\frac{m}{s_{2}}\in S^{-1}(M$-$rad(N))=S^{-1}M$-$rad(S^{-1}N)$ and we are done.

(2) Let $r\in R$, $m\in M$ and $rm\in N$. Then $\frac{r}{1}\frac{m}{1}\in
S^{-1}N$ which implies that $\frac{r}{1}\in(J(S^{-1}R)S^{-1}M:S^{-1}%
M)=S^{-1}(J(R)M:M)$ or $\frac{m}{1}\in S^{-1}M$-$rad(S^{-1}N)=S^{-1}%
(M$-$rad(N)).$ Hence, either $ur\in(J(R)M:M)$ for some $u\in S$ or $vm\in
M$-$rad(N)$ for some $v\in S$. Thus, our assumptions imply that either
$r\in(J(R)M:M)$ or $m\in M$-$rad(N)$ as needed.
\end{proof}

Following \cite{Hosein}, a proper submodule $N$ of an $R$-module $M$ is called
quasi primary if whenever $r\in R$ and $m\in M$ such that $rm\in N$, then
either $r\in\sqrt{N:M}$ or $m\in M$-$rad(N)$.

\begin{proposition}
\label{5}If $N$ is a quasi-primary submodule of an $R$-module $M$ such that
$(N:M)\subseteq J(R)$, then $N$ is a quasi $J$-submodule of $M.$
\end{proposition}

\begin{proof}
Suppose $N$ is quasi-primary and $(N:M)\subseteq J(R)$. Let $r\in R$ and $m\in
M$ such that $rm\in N$ and $r\notin(J(R)M:M).$ Then $r\notin J(R)$ and so by
assumption $r\notin\sqrt{N:M}$. It follows that $m\in M$-$rad(N)$ as needed.
\end{proof}

Note that if $(N:M)\nsubseteq J(R)$, then the above proposition need not be
true. For example, consider the submodule $N=\left\langle 2\right\rangle $ of
the $%
%TCIMACRO{\U{2124} }%
%BeginExpansion
\mathbb{Z}
%EndExpansion
$-module $%
%TCIMACRO{\U{2124} }%
%BeginExpansion
\mathbb{Z}
%EndExpansion
$. Then $(N:M)=\left\langle 2\right\rangle \nsubseteq J(%
%TCIMACRO{\U{2124} }%
%BeginExpansion
\mathbb{Z}
%EndExpansion
)$. Moreover, $N$ is primary (and so quasi-primary) which is clearly not a
quasi $J$-submodule. In view of \cite[Theorem 2]{Haniece}, we have:

\begin{corollary}
If $N$ is a quasi-primary submodule of an $R$-module $M$ such that $(N:M)$ is
a quasi $J$-ideal of $R$, then $N$ is a quasi $J$-submodule of $M.$
\end{corollary}

Following \cite{Rib}, a submodule $N$ of an $R$-module $M$ is called a pure
submodule if $rM\cap N=rN$ for each $r\in R$. Moreover, $N$ is called
divisible if $rN=N$ for each $r\in Reg(R)$, the set of regular elements in $R$.

\begin{proposition}
Let $N$ be a divisible $J$-submodule of an $R$-module $M$ with
$(J(R)M:M)\subseteq Reg(R)$. Then $N$ is pure in $M$.
\end{proposition}

\begin{proof}
It is clear that $rN\subseteq rM\cap N$ for each $r\in R$. Let $r\in R$ and
let $n\in rM\cap N$. Then $n=rm\in N$ for some $m\in M$. If $r\in(J(R)M:M)$,
then by assumption, $rM\cap N\subseteq N=rN$. If $r\notin(J(R)M:M)$, then
$m\in N$ since $N$ is a $J$-submodule of $M$ and so $n=rm\in rN$. Thus,
$rM\cap N=rN$ and $N$ is pure in $M$.
\end{proof}

\bigskip Synonymously to the Prime Avoidance Lemma for prime submodules, we have:

\begin{proposition}
\label{av}Let $M$ be an $R$-module such that $J(R)=(J(R)M:M)$ is a quasi
$J$-ideal of $R$. Let $N,N_{1},N_{2},...,N_{k}$ be submodules of $M$ where
$N\subseteq%
%TCIMACRO{\dbigcup \limits_{i=1}^{k}}%
%BeginExpansion
{\displaystyle\bigcup\limits_{i=1}^{k}}
%EndExpansion
N_{i}$. Suppose that $N_{j}$ is a $J$-submodule (resp. quasi $J$-submodule)
with $(N_{i}:M)\nsubseteq J(R)$ for all $i\neq j$. If $N\nsubseteq%
%TCIMACRO{\dbigcup \limits_{i\neq j}^{k}}%
%BeginExpansion
{\displaystyle\bigcup\limits_{i\neq j}^{k}}
%EndExpansion
N_{i}$, then $N\subseteq N_{j}$ (resp. $N\subseteq M$-$rad(N_{j})$).
\end{proposition}

\begin{proof}
Without loss of generality, assume that $j=k$. First, we show that
$N\cap\left(
%TCIMACRO{\dbigcap \limits_{i=1}^{k-1}}%
%BeginExpansion
{\displaystyle\bigcap\limits_{i=1}^{k-1}}
%EndExpansion
N_{i}\right)  \subseteq N_{k}$. Let $x\in N\cap\left(
%TCIMACRO{\dbigcap \limits_{i=1}^{k-1}}%
%BeginExpansion
{\displaystyle\bigcap\limits_{i=1}^{k-1}}
%EndExpansion
N_{i}\right)  $. Since $N\nsubseteq%
%TCIMACRO{\dbigcup \limits_{i=1}^{k-1}}%
%BeginExpansion
{\displaystyle\bigcup\limits_{i=1}^{k-1}}
%EndExpansion
N_{i}$, there exists an element $m\in N_{k}$ but $m\notin%
%TCIMACRO{\dbigcup \limits_{i=1}^{k-1}}%
%BeginExpansion
{\displaystyle\bigcup\limits_{i=1}^{k-1}}
%EndExpansion
N_{i}$. Then clearly $m+x\in N\backslash\left(
%TCIMACRO{\dbigcup \limits_{i=1}^{k-1}}%
%BeginExpansion
{\displaystyle\bigcup\limits_{i=1}^{k-1}}
%EndExpansion
N_{i}\right)  $. Hence, $m+x\in N_{k}$ and so $x\in N_{k}$. Now, since
$(N_{i}:M)\nsubseteq J(R)$ for all $i\neq k$, there is an element $r_{i}%
\in(N_{i}:M)\backslash J(R)$ for all $i\neq k$. Put $r=%
%TCIMACRO{\dprod \limits_{i=1}^{k-1}}%
%BeginExpansion
{\displaystyle\prod\limits_{i=1}^{k-1}}
%EndExpansion
r_{i}$. Since $J(R)$ is a prime ideal of $R$, \cite[Corollary 2]{Haniece},
then $r\notin J(R)$. Put $I=%
%TCIMACRO{\dbigcap \limits_{i=1}^{k-1}}%
%BeginExpansion
{\displaystyle\bigcap\limits_{i=1}^{k-1}}
%EndExpansion
(N_{i}:M)$. Then $I\nsubseteq J(R)=(J(R)M:M)$ and $IN\subseteq N\cap\left(
%TCIMACRO{\dbigcap \limits_{i=1}^{k-1}}%
%BeginExpansion
{\displaystyle\bigcap\limits_{i=1}^{k-1}}
%EndExpansion
N_{i}\right)  \subseteq N_{k}$. Since $N_{k}$ is a (quasi) $J$-submodule, we
conclude that $N\subseteq N_{j}$ (resp. $N\subseteq M$-$rad(N_{j})$) by
Proposition \ref{eq1}.\bigskip
\end{proof}

For an $R$-module $M$, consider the set of all zero divisors on $M$,
$Z(M)=\left\{  r\in R:rm=0\text{ for some }0\neq m\in M\right\}  $. Following
\cite{Anderson3}, we call an $R$-module $M$ presimplifiable if whenever $r\in
R$ and $m\in M$ such that $rm=m$, then $m=0$ or $r\in U(R)$. Equivalently, $M$
is presimplifiable if and only if $Z(M)\subseteq J(R)$. We recall that the
prime radical of an $R$-module $M$ is the intersection of all prime submodules
in $M$ and is denoted by $Nil(M)$. It is known that for a submodule $N$ of
$M$, there is a one to one correspondence between the prime submodules of
$M/N$ and those of $M$ containing $N$. Hence, we get $m\in M$-$rad(N)$
$\Longleftrightarrow$ $m+N\in Nil(M/N)$.

More generally, let $NZ(M)=\left\{  r\in R:rm=0\text{ for some }m\notin
Nil(M)\right\}  $. Next, we define some generalizations of presimplifiable modules.

\begin{definition}
Let $M$ be an $R$-module.

\begin{enumerate}
\item $M$ is called quasi presimplifiable if $NZ(M)\subseteq J(R)$.

\item $M$ is called $J$-presimplifiable if $Z(M)\subseteq(J(R)M:M)$.

\item $M$ is called a quasi $J$-presimplifiable if $NZ(M)\subseteq(J(R)M:M)$.
\end{enumerate}
\end{definition}

\begin{example}
\label{ep}

\begin{enumerate}
\item Consider the $%
%TCIMACRO{\U{2124} }%
%BeginExpansion
\mathbb{Z}
%EndExpansion
$-module $M=%
%TCIMACRO{\U{2124} }%
%BeginExpansion
\mathbb{Z}
%EndExpansion
_{p}$ for a prime integer $p$. Then $Z(M)=\left\langle p\right\rangle
=(0:M)=(J(%
%TCIMACRO{\U{2124} }%
%BeginExpansion
\mathbb{Z}
%EndExpansion
)M:M)$. Thus, $M$ is a $J$-presimplifiable module that is not presimplifiable.

\item The $%
%TCIMACRO{\U{2124} }%
%BeginExpansion
\mathbb{Z}
%EndExpansion
$-module $M=%
%TCIMACRO{\U{2124} }%
%BeginExpansion
\mathbb{Z}
%EndExpansion
_{4}$ is a quasi $J$-presimplifiable that is not $J$-presimplifiable. Indeed,
we have $Z(M)=2%
%TCIMACRO{\U{2124} }%
%BeginExpansion
\mathbb{Z}
%EndExpansion
$ and $NZ(M)=4%
%TCIMACRO{\U{2124} }%
%BeginExpansion
\mathbb{Z}
%EndExpansion
=(J(%
%TCIMACRO{\U{2124} }%
%BeginExpansion
\mathbb{Z}
%EndExpansion
)M:M)$.

\item Consider the $%
%TCIMACRO{\U{2124} }%
%BeginExpansion
\mathbb{Z}
%EndExpansion
(+)%
%TCIMACRO{\U{2124} }%
%BeginExpansion
\mathbb{Z}
%EndExpansion
_{2}$-module $M=%
%TCIMACRO{\U{2124} }%
%BeginExpansion
\mathbb{Z}
%EndExpansion
(+)%
%TCIMACRO{\U{2124} }%
%BeginExpansion
\mathbb{Z}
%EndExpansion
_{2}$. Then $M$ is a quasi presimplifiable module that is not presimplifiable,
see \cite[Example 5]{Haniece}.
\end{enumerate}
\end{example}

In the following theorem, we characterize $J$-submodules (resp. quasi
$J$-submodules) in terms of $J$-presimplifiable (resp. quasi $J$%
-presimplifiable) modules.

\begin{theorem}
\label{tp}Let $N$ be a submodule of an R-module $M$ with $N\subseteq J(R)M$.
Then $N$ is a quasi $J$-submodule (resp. $J$-submodule) if and only if $M/N$
is a non-zero quasi $J$-presimplifiable (resp. $J$-presimplifiable) $R$-module.
\end{theorem}

\begin{proof}
Suppose $N$ is a quasi $J$-submodule. Let $r\in NZ(M/N))$ and choose
$m+N\notin Nil(M/N)$ such that $r(m+N)=N$. Then $rm\in N$ and $m\notin
M$-$rad(N)$ since otherwise, if $m\in M$-$rad(N)$, then $m+N\in Nil(M/N)$, a
contradiction. Since $N$ is a quasi $J$-submodule, then $r\in
(J(R)M:M)=((J(R)M)/N:M/N)=(J(R)(M/N):M/N)$ as needed. Conversely, suppose
$M/N$ is a non-zero quasi $J$-presimplifiable and let $r\in R$ and $m\in M$
such that $rm\in N$ and $r\notin(J(R)M:M)=(J(R)(M/N):M/N)$. Then $r.(m+N)=N$
and $r\notin NZ(M/N)$. Therefore, we must have $m+N\in Nil(M/N)$ and then
$m\in M$-$rad(N)$. The proof of the $J$-submodule case is similar.
\end{proof}

\begin{corollary}
Let $N$ be a submodule of an $R$-module $M$ such that $N\subseteq J(R)M$ and
$(J(R)M:M)=J(R)$. Then the following statements are equivalent:

\begin{enumerate}
\item $N$ is a (quasi) $J$-submodule.

\item $M/N$ is a non-zero (quasi) $J$-presimplifiable.

\item $M/N$ is a non-zero (quasi) presimplifiable.
\end{enumerate}
\end{corollary}

Recall that for an $R$-module $M$, $T(M)=\left\{  m\in M:ann_{R}%
(m)\neq0\right\}  $. It is clear that if $N$ is an $sr$-submodule of $M$, then
$N\subseteq T(M)$.

\begin{proposition}
Let $M$ be an $R$-module.

(1) If $M$ is $J$-presimplifiable and $N$ is an $r$-submodule of $M$, then $N$
is a (quasi) $J$-submodule of $M$.

(2) If $N$ is an $sr$-submodule of $M$ with $T(M)=N\subseteq J(R)M$, then $N$
is a (quasi) $J$-submodule of $M$.
\end{proposition}

\begin{proof}
(1) Suppose $N$ is an $r$-submodule and let $r\in R$ and $m\in M$ such that
$rm\in N$ and $r\notin(J(R)M:M)$. Since $M$ is $J$-presimplifiable, then
$r\notin Z(M)$ and so clearly $ann_{M}(r)=0$. Hence, $m\in N\subseteq
M$-$rad(N)$ as $N$ is an $r$-submodule and we are done.

(2) Suppose $N$ is an $sr$-submodule of $M$ with $N=T(M)$. Let $r\in R$ and
$m\in M$ such that $rm\in N$. If $m\notin M$-$rad(N)$, then $m\notin T(M)$ and
so $ann_{R}(m)=0$. By assumption, we get $r\in(N:M)\subseteq(J(R)M:M)$ and $N$
is a (quasi) $J$-submodule of $M$.
\end{proof}

\section{Quasi $J$-submodule in multiplication modules}

In this section we study quasi $J$-submodules in some special types of
modules. We give several properties and characterizations of quasi
$J$-submodules in finitely generated faithful multiplication modules.
Moreover, we determine conditions on a submodule $N$ of $M$ and an ideal $I$
of $R$ for which $I(+)N$ is a quasi $J$-ideal in $R(+)M$.

We start by the following lemma.

\begin{lemma}
\cite{Smith}\label{9}Let $M$ be a finitely generated faithful multiplication
$R$-module, $N$ be a proper submodule of $M$ and $I$ be an ideal of $R$. Then

\begin{enumerate}
\item $M$-$rad(N)=\sqrt{N:M}M$.

\item $IM:M=I$.

\item $(IN:M)=I(N:M)$.
\end{enumerate}
\end{lemma}

In view of the properties in Lemma \ref{9}, we give the following
characterizations of quasi $J$-submodules of finitely generated faithful
multiplication modules.

\begin{theorem}
\label{1}Let $I$ be an ideal of a ring $R$ and $N$ be a submodule of a
finitely generated faithful multiplication $R$-module $M$. Then
\end{theorem}

\begin{enumerate}
\item $I$ is a quasi $J$-ideal of $R$ if and only if $IM$ is a quasi
$J$-submodule of $M$.

\item $N$ is a quasi $J$-submodule of $M$ if and only if $(N:M)$ is a quasi
$J$-ideal of $R$.

\item $N$ is a quasi $J$-submodule of $M$ if and only if $N=IM$ for some quasi
$J$-ideal of $R$.

\item If $I$ is a quasi $J$-ideal of $R$ and $N$ is a quasi $J$-submodule of
$M$, then $IN$ is a quasi $J$-submodule of $M$.
\end{enumerate}

\begin{proof}
(1) Suppose $I$ is a quasi $J$-ideal of $R$. If $IM=M$, then $I=(IM:M)=R$, a
contradiction. Thus, $IM$ is proper in $M$. Now, let $r\in R$ and $m\in M$
such that $rm\in IM$ and $r\notin(J(R)M:M)=J(R)$. Then
$r((m):M)=((rm):M)\subseteq(IM:M)\subseteq(\sqrt{I}M:M)=\sqrt{I}$. As $I$ is a
quasi $J$-ideal of $R$, we conclude that $((m):M)\subseteq\sqrt{I}$. Thus,
$m\in((m):M)M\subseteq\sqrt{I}M=M$-$rad(IM)$. Conversely, suppose $IM$ is a
quasi $J$-submodule of $M$. Then clearly $I$ is proper in $R$. Let $a,b\in R$
such that $ab\in I$ and $a\notin J(R)=(J(R)M:M)$. Since $abM\in IM$ and $IM$
is a quasi $J$-submodule, then $bM\subseteq M$-$rad(IM)=\sqrt{I}M$. Therefore,
$b\in\sqrt{I}M:M=\sqrt{I}$ and $I$ is a quasi $J$-ideal of $R$.

(2) Follows by (1) since $N=(N:M)M$.

(3) Follows by choosing $I=(N:M)$ and using (2).

(4) Suppose $I$ is a quasi $J$-ideal of $R$ and $N$ is a quasi $J$-submodule
of $M$. Now, $(N:M)$ is a quasi $J$-ideal of $R$ by (2) and so $I(N:M)$ is
also quasi $J$-ideal by \cite[Proposition 4]{Haniece}. Moreover, $IN=I(N:M)M$
is proper in $M$ since otherwise, $I(N:M)=R$, a contradiction. By using (1),
we conclude that $IN$ is a quasi $J$-submodule of $M$.
\end{proof}

However, the equivalence in (2) of Theorem \ref{1} can not be achieved if $M$
is not finitely generated faithful multiplication. For example, consider the $%
%TCIMACRO{\U{2124} }%
%BeginExpansion
\mathbb{Z}
%EndExpansion
$-module $M=$ $%
%TCIMACRO{\U{2124} }%
%BeginExpansion
\mathbb{Z}
%EndExpansion
\times%
%TCIMACRO{\U{2124} }%
%BeginExpansion
\mathbb{Z}
%EndExpansion
$ and the submodule $N=2%
%TCIMACRO{\U{2124} }%
%BeginExpansion
\mathbb{Z}
%EndExpansion
\times0$ of $M$. Then clearly $(N:M)=0$ is a quasi $J$-ideal of $%
%TCIMACRO{\U{2124} }%
%BeginExpansion
\mathbb{Z}
%EndExpansion
$, but $N$ is not a quasi $J$-submodule of $M$. In fact, $2.(1,0)\in N$ but
neither $2\in(J(%
%TCIMACRO{\U{2124} }%
%BeginExpansion
\mathbb{Z}
%EndExpansion
)M:M)=0$ nor $(1,0)\in M$-$rad(N)=N$.

\begin{proposition}
Let $N$ be a submodule of a faithful multiplication $R$-module $M$. Let $I$ be
a finitely generated faithful multiplication ideal of $R$. Then

\begin{enumerate}
\item If $IN$ is a $J$-submodule of $M$, then either $I$ is a $J$-ideal of $R$
or $N$ is a $J$-submodule of $M$.

\item If $\sqrt{I}$ is a finitely generated multiplication ideal of $R$ and
$\sqrt{I}N$ is a quasi $J$-submodule of $M$, then either $I$ is a quasi
$J$-ideal of $R$ or $N$ is a quasi $J$-submodule of $M$.
\end{enumerate}
\end{proposition}

\begin{proof}
(1) If $N=M$, then $(IN:M)=I(N:M)=IR=I$ is a $J$-ideal of $R$ by
\cite[Corollary 3.4]{Hani}. Suppose $N\varsubsetneq M$. Since $I$ is finitely
generated faithful multiplication, we have $N=(IN:_{M}I)$, \cite[Lemma
2.4]{Majed}. Hence, one can easily verify that $(N:M)=((IN:_{M}%
I):M)=(I(N:M):I)$. Let $a,b\in R$ such that $ab\in(N:M)$ and $a\notin J(R)$.
Then $Iab\subseteq I(N:M)=(IN:M)$ and so $Ib\subseteq I(N:M)$ as $(IN:M)$ is a
$J$-ideal. It follows that $b\in(I(N:M):I)=(N:M)$ and so $(N:M)$ is a
$J$-ideal of $R$. The result follows again by \cite[Corollary 3.4]{Hani}.

(2) If $N=M$, then $\sqrt{I}=\sqrt{I}(N:M)=(\sqrt{I}N:M)$ is a quasi $J$-ideal
of $R$ by (2) of Theorem \ref{1}. It follows clearly that $I$ is a quasi
$J$-ideal. Suppose $N\varsubsetneq M$ and note again by \cite[Lemma
2.4]{Majed} that $M$-$rad(N)=(\sqrt{I}(M$-$rad(N)):_{M}\sqrt{I})$. Let $rm\in
N$ and $r\notin J(R)$ for $r\in R$ and $m\in M$. Then $\sqrt{I}rm\subseteq
\sqrt{I}N$ and so $\sqrt{I}m\subseteq M$-$rad(\sqrt{I}N)=\sqrt{I}(M$%
-$rad(N))$. It follows that $m\in(\sqrt{I}(M$-$rad(N)):_{M}\sqrt{I}%
)=M$-$rad(N)$ and $N$ is a quasi $J$-submodule of $M$.
\end{proof}

\begin{theorem}
\label{3}Let $N$ be a proper submodule of a finitely generated faithful
multiplication $R$-module $M$. The following are equivalent:
\end{theorem}

\begin{enumerate}
\item $N$ is a quasi $J$-submodule.

\item $M$-$rad(N)$ is a quasi $J$-submodule.

\item $M$-$rad(N)$ is a $J$-submodule.

\item $(M$-$rad(N):_{M}\left\langle r\right\rangle )=M$-$rad(N)$ for all
$r\notin J(R)$.
\end{enumerate}

\begin{proof}
(1)$\Rightarrow$(2) Suppose $N$ is a quasi $J$-submodule and let $r\in R$ and
$m\in M$ such that $rm\in M$-$rad(N)$ and $r\notin(J(R)M:M)=J(R)$. Then
$rm\in\sqrt{N:M}M$ and so $r((m):M)=((rm):M)\subseteq(\sqrt{N:M}%
M:M)=\sqrt{N:M}$. Since $(N:M)$ is a quasi $J$-ideal by Theorem \ref{1}, then
$((m):M)\subseteq\sqrt{N:M}$. It follows that $m\in((m):M)M\subseteq\sqrt
{N:M}M=M$-$rad(N)$.

(2)$\Rightarrow$(3) It is straightforward as $M$-$rad(M$-$rad(N))=M$-$rad(N).$

(3)$\Rightarrow$(4) Let $m\in(M$-$rad(N):_{M}\left\langle r\right\rangle )$.
Then $rm\in M$-$rad(N)$ with $r\notin(J(R)M:M)$ and so $m\in M$-$rad(N)$ by
our assumption (2). The other inclusion is clear.

(4)$\Rightarrow$(1) Suppose that $rm\in N$ and $r\notin(J(R)M:M).$ Then $rm\in
M$-$rad(N)$ and $r\notin J(R)$ which imply that $m\in(M$-$rad(N):_{M}%
<r>)=M$-$rad(N).$ Thus, $N$ is a quasi $J$-submodule.
\end{proof}

In general the equivalences in Theorem \ref{3} need not be true if $M$ is not
finitely generated faithful multiplication. For example, while $\left\langle
\bar{0}\right\rangle $ is a quasi J-submodule in the $%
%TCIMACRO{\U{2124} }%
%BeginExpansion
\mathbb{Z}
%EndExpansion
$-module $M=%
%TCIMACRO{\U{2124} }%
%BeginExpansion
\mathbb{Z}
%EndExpansion
_{4}$, $M$-$rad(\left\langle \bar{0}\right\rangle )=\left\langle \bar
{2}\right\rangle $ is not a $J$-submodule since for example $2\cdot\bar{1}%
\in\left\langle \bar{2}\right\rangle $ while $2\notin(\bar{0}:$ $%
%TCIMACRO{\U{2124} }%
%BeginExpansion
\mathbb{Z}
%EndExpansion
_{4})$ and $\bar{1}\notin\left\langle \bar{2}\right\rangle $.

In view of Theorem \ref{1} and Theorem \ref{3}, we also have:

\begin{proposition}
\label{7}Let $M$ be a finitely generated faithful multiplication $R$-module.
For any submodule $N$ of $M$, the following statements are equivalent.
\end{proposition}

\begin{enumerate}
\item $N$ is a quasi $J$-submodule of $M$..

\item $\sqrt{(N:M)}$ is a $J$-ideal of $R$.

\item $\sqrt{(N:M)}$ is a quasi $J$-ideal of $R$.

\item $(N:M)$ is a quasi $J$-ideal of $R$.
\end{enumerate}

\begin{proposition}
Let $N$, $K$ and $L$ be submodules of an $R$-module $M$ and $I$ be an ideal of
$R$ with $I\nsubseteq(J(R)M:M)$. Then
\end{proposition}

\begin{enumerate}
\item If $K$ and $L$ are quasi $J$-submodules of $M$ with $IK=IL$, then
$M$-$rad(K)=M$-$rad(L).$

\item If $IN$ is a quasi $J$-submodule of a finitely generated faithful
multiplication module $M$, then $N$ is a quasi $J$-submodule of $M$.
\end{enumerate}

\begin{proof}
(1) Suppose that $IK=IL$. Then $IK\subseteq L$ and $I\nsubseteq(J(R)M:M)$
imply that $K\subseteq M$-$rad(L)$ by Proposition \ref{eq1} and $M$%
-$rad(K)\subseteq M$-$rad(M$-$rad(L))=M$-$rad(L)$. Similarly, we conclude that
$M$-$rad(L)\subseteq M$-$rad(K)$, so the equality holds.

(2) Let $IN$ be a quasi $J$-submodule of $M$. Since $IN\subseteq$ $IN$ and
$I\nsubseteq(J(R)M:M)$, we conclude that $N\subseteq$ $M$-$rad(IN).$ Hence,
$M$-$rad(N)=$ $M$-$rad(IN)$ which is\ clearly a $J$-submodule of $M$ by
Theorem \ref{3}. It follows again by Theorem \ref{3}, that $N$ is a quasi
$J$-submodule of $M$.
\end{proof}

\begin{lemma}
\cite{Majed}\label{2}Let $I$ be a faithful multiplication ideal of a ring $R$
and $M$ be a faithful multiplication $R$-module. Then

\begin{enumerate}
\item For every submodule $N$ of $IM$, we have $(IM)$-$rad(N)=I(M$%
-$rad(N:_{M}I))$.

\item If $N$ is a submodule of $M$ and $I$ is finitely generated, then
$N=(IN:_{M}I)$.
\end{enumerate}
\end{lemma}

\begin{proposition}
Let $I$ be a faithful multiplication ideal of a ring $R$ and $M$ be a faithful
multiplication $R$-module. If $N$ is a quasi $J$-submodule of $IM$, then
$(N:_{M}I)$ is a quasi $J$-submodule of $M$. Moreover, the converse is true if
$R$ is quasi-local.
\end{proposition}

\begin{proof}
Suppose $N$ is a quasi $J$-submodule of $IM$. Then $N\varsubsetneq IM$ and so
clearly,$(N:_{M}I)\varsubsetneq M$. Let $r\in R$ and $m\in M$ such that
$rm\in(N:_{M}I)$ and $r\notin(J(R)M:M)$. Then $rmI\subseteq N$ and by Lemma
\ref{2} $r\notin(J(R)IM:IM)$, hence $mI\subseteq(IM)$-$rad(N)=I(M$%
-$rad(N:_{M}I))$. It follows by Lemma \ref{2} that $m\in(I(M$-$rad(N:_{M}%
I)):I)=M$-$rad(N:_{M}I)$. Therefore, $(N:_{M}I)$ is a quasi $J$-submodule of
$M$. Now, suppose $R$ is quasi-local and $(N:_{M}I)$ is a quasi $J$-submodule
of $M$. Then clearly, $N$ is proper in $M$ and $I=\left\langle a\right\rangle
$ is principal, see \cite{Anderson1}. Let $r\in R$ and $m\in IM$ such that
$rm\in N$ and $r\notin(J(R)IM:IM)$. Since $I=\left\langle a\right\rangle $,
then we may write $m=am_{1}$ for some $m_{1}\in M$. Hence, $rm_{1}\in
(N:_{M}I)$ and clearly $r\notin(J(R)M:M)$. So, $m_{1}\in M$-$rad((N:_{M}I))$
as $(N:_{M}I)$ is a quasi $J$-submodule of $M$. Again by Lemma \ref{2}, we
have $m=am_{1}\in I(M$-$rad(N:_{M}I))=(IM)$-$rad(N)$ and the result follows.
\end{proof}

A submodule $N$ of an $R$-module $M$ is said to be small (or superfluous) in
$M$, abbreviated $N\ll M$, in case for any submodule $K$ of $M$, $N+K=M$
implies $K=M.$

\begin{proposition}
Every quasi $J$-submodules of a finitely generated faithful multiplication
$R$-module is small.
\end{proposition}

\begin{proof}
Let $N$ be a quasi $J$-submodule of an $R$-module $M$ and $K$ be a submodule
of $M$ with $N+K=M$. Then clearly $(N:M)+(K:M)=(N+K:M)=R$ and $(N:M)$ is a
quasi $J$-ideal of $R$ by Theorem \ref{1}. Hence $(K:M)=R$ by
\cite[Proposition 4]{Haniece} and so $K=M$ as desired.
\end{proof}

Let $M$ be an $R$-module and $N$ be a submodule of $M$. We denote the
intersection of all maximal submodules of $M$ by $J(M)$. In particular, by
$J(N)$, we denote the intersection of all maximal submodules of $M$ containing
$N$. It is well known that If $M$ is finitely generated faithful
multiplication, then $J(M)=J(R)M$, \cite{Bast}. In particular, we have
$J(N)=J(N:M)M$.

In the next two theorems, we obtain more characterizations for quasi
$J$-submodules in finitely generated faithful multiplication modules.

\begin{theorem}
\label{J(I)}Let $N$ be a $J$-submodule of a finitely generated faithful
multiplication $R$-module $M.$ Then the following statements are equivalent:
\end{theorem}

\begin{enumerate}
\item $N$ is a quasi $J$-submodule of $M.$

\item $N\subseteq J(M)$ and if whenever $r\in R$ and $m\in M$ with $rm\in N$
and $r\notin(J(N):M)$, then $m\in M$-$rad(N)$.
\end{enumerate}

\begin{proof}
(1)$\Rightarrow$(2) Suppose $N$ is a quasi $J$-submodule. since $(N:M)$ is a
quasi $J$-ideal by Theorem \ref{1}, then $(N:M)\subseteq J(R)$, \cite[Theorem
2]{Haniece}. Thus, $N=(N:M)M\subseteq J(R)M=J(M)$. Moreover, let $r\in R$ and
$m\in M$ with $rm\in N$ and $r\notin(J(N):M)$. Then $r\notin
(J(M):M)=(J(R)M:M)$ as clearly $J(M)\subseteq J(N)$ and so $m\in M$-$rad(N)$
by assumption.

(2)$\Rightarrow$(1) If $N=M$, then $J(M)=M$, a contradiction. Let $r\in R$ and
$m\in M$ with $rm\in N$ and $r\notin(J(R)M:M)=J(R)$. Since $N\subseteq J(M)$,
then one can easily see that $J(N)\subseteq J(J(M))=J(M)$ and so $J(M)=J(N)$.
Thus, $r\notin(J(N):M)$ and so $m\in M$-$rad(N)$ as required.
\end{proof}

Recall that If $M$ is a multiplication $R$-module and $N=IM$, $K=JM$ are two
submodules of $M$, then the product $NK$ of $N$ and $K$ is defined as
$NK=(IM)(JM)=(IJ)M$. In particular, if $m_{1},m_{2}\in M$, then $m_{1}%
m_{2}=\left\langle m_{1}\right\rangle \left\langle m_{2}\right\rangle $.

\begin{proposition}
Let $M$ be a finitely generated faithful multiplication $R$-module and $N$ a
proper submodule of $M$. Then $N$ is a quasi $J$-submodule of $M$ if and only
if whenever $K$ and $L$ are submodules of $M$ with $KL\subseteq N$, then
$K\subseteq J(M)$ or $L\subseteq M$-$rad(N).$
\end{proposition}

\begin{proof}
Suppose $K=IM$ and $L=JM$ for some ideals $I$ and $J$ of $R$ and $KL\subseteq
N$. Then $I(JM)\subseteq N$ and so $I\subseteq(J(R)M:M)=J(R)$ or
$L=JM\subseteq M$-$rad(N)$ by Proposition \ref{eq1}. Thus, $K\subseteq
J(R)M=J(M)$ or $L\subseteq M$-$rad(N)$. Conversely, let $A\nsubseteq
(J(R)M:M)=J(R)$ be an ideal of $R$ and $L$ be a submodule of $M$ with
$AL\subseteq N.$ Then the result follows by putting $K=AM$ and using again
Proposition \ref{eq1}.
\end{proof}

\begin{corollary}
Let $N$ be a proper submodule of a finitely generated faithful multiplication
$R$-module $M$. Then $N$ is a quasi $J$-submodule of $M$ if and only if
whenever $m_{1},m_{2}\in M$ such that $m_{1}m_{2}\in N$, then $m_{1}\in J(M)$
or $m_{2}\in M$-$rad(N)$.
\end{corollary}

\begin{proposition}
\label{(N:S)}Let $M$ be a finitely generated faithful multiplication
$R$-module and let $S$ be a subset of $R$ with $S\nsubseteq J(R)$. If $N$ is a
quasi $J$-submodule of $M$, then $(N:_{M}S)$ is a quasi $J$-submodule of $M.$
\end{proposition}

\begin{proof}
First, we prove that $(N:_{M}S)$ is proper in $M$. Suppose $(N:_{M}S)=M$ and
let $m\in M$. Then $Sm\subseteq N$ and since $N$ is a quasi $J$-submodule, we
get $m\in M$-$rad(N)$. Thus, $M=M$-$rad(N)=N$, a contradiction. Now, similar
to the proof of (4) in Theorem \ref{3}, one can prove that $(M-rad(N):_{M}%
S)=M-rad(N)$. Suppose that $rm\in(N:_{M}S)$ and $r\notin J(R).$ Then
$rSm\subseteq N$ and so $Sm\subseteq M-rad(N)$ as $N$ is a quasi
$J$-submodule. It follows that $m\in(M-rad(N):_{M}S)=M-rad(N)\subseteq$
$M-rad(N:_{M}S)$ as required.
\end{proof}

A proper submodule $N$ of an $R$-module $M$ is called a maximal quasi
$J$-submodule if there is no quasi $J$-submodule which contains $N$ properly.

\begin{theorem}
\label{max}Every quasi $J$-submodule of an $R$-module $M$ is contained in a
maximal quasi $J$-submodule of $M$. Moreover, if $M$ is finitely generated
faithful multiplication, then a maximal quasi $J$-submodule of $M$ is a $J$-submodule.
\end{theorem}

\begin{proof}
Suppose that $N$ is a quasi $J$-submodule of $M$ and Set $\Omega=\{N_{\alpha
}:N_{\alpha}$ is a quasi $J$-submodule of $M$, $\alpha\in\Lambda\}$. Then
$\Omega\neq\emptyset.$ Let $N_{1}\subseteq N_{2}\subseteq\cdots$ be any chain
in $\Omega$. We show that $%
%TCIMACRO{\dbigcup \limits_{i=1}^{\infty}}%
%BeginExpansion
{\displaystyle\bigcup\limits_{i=1}^{\infty}}
%EndExpansion
N_{i}$ is a quasi $J$-submodule of $M$. Suppose $rm\in%
%TCIMACRO{\dbigcup \limits_{i=1}^{\infty}}%
%BeginExpansion
{\displaystyle\bigcup\limits_{i=1}^{\infty}}
%EndExpansion
N_{i}$ for $r\in R$, $m\in M$ and $r\notin(J(R)M:M)$. Then $rm\in N_{j}$ for
some $j\in%
%TCIMACRO{\U{2115} }%
%BeginExpansion
\mathbb{N}
%EndExpansion
$ which implies that $m\in M$-$rad(N_{j})\subseteq M-rad(%
%TCIMACRO{\dbigcup \limits_{i=1}^{\infty}}%
%BeginExpansion
{\displaystyle\bigcup\limits_{i=1}^{\infty}}
%EndExpansion
N_{i})$. Since also $%
%TCIMACRO{\dbigcup \limits_{i=1}^{\infty}}%
%BeginExpansion
{\displaystyle\bigcup\limits_{i=1}^{\infty}}
%EndExpansion
N_{i}$ is clearly proper, then $%
%TCIMACRO{\dbigcup \limits_{i=1}^{\infty}}%
%BeginExpansion
{\displaystyle\bigcup\limits_{i=1}^{\infty}}
%EndExpansion
N_{i}$ is a quasi $J$-submodule which is an upper bound of the chain
$\{N_{i}:i\in%
%TCIMACRO{\U{2115} }%
%BeginExpansion
\mathbb{N}
%EndExpansion
\}$. By Zorn's Lemma, $\Omega$ has a maximal element which is a maximal quasi
$J$-submodule of $M$. Now, let $K$ be a maximal quasi $J$-submodule of $M$.
Suppose that $rm\in K$ and $r\notin J(R)$. Then $(K:_{M}r)$ is also a quasi
$J$-submodule of $M$ by Proposition \ref{(N:S)}. Thus, the maximality of $K$
implies that $m\in(K:_{M}r)=K$ and we are done.
\end{proof}

In view of Theorem \ref{max}, we have the following.

\begin{corollary}
\label{J}Let $M$ be a finitely generated faithful multiplication $R$-module.
Then the following statements are equivalent:

\begin{enumerate}
\item $J(M)$ is a $J$-submodule of $M$.

\item $J(M)$ is a quasi $J$-submodule of $M$.

\item $J(M)$ is a prime submodule of $M$.
\end{enumerate}
\end{corollary}

\begin{proof}
(1)$\Rightarrow$ (2) is clear.

(2)$\Rightarrow$ (1) It follows since $J(M)$ is the unique maximal quasi
$J$-submodule of $M$ by Theorem \ref{J(I)}.

(2)$\Leftrightarrow$ (3) Since $M-rad(J(M))=J(M)$, the claim is clear.
\end{proof}

In the following proposition, we prove that $J$-submodule property passes to a
finite intersection and product.

\begin{proposition}
Let $M$ be a multiplication $R$-module and $N_{1},N_{2},...,N_{k}$ be quasi
$J$-submodules of $M$. Then so are $%
%TCIMACRO{\dbigcap \limits_{i=1}^{k}}%
%BeginExpansion
{\displaystyle\bigcap\limits_{i=1}^{k}}
%EndExpansion
N_{i}$ and $\prod\limits_{i=1}^{k}N_{i}$.
\end{proposition}

\begin{proof}
Suppose that $rm\in%
%TCIMACRO{\dbigcap \limits_{i=1}^{k}}%
%BeginExpansion
{\displaystyle\bigcap\limits_{i=1}^{k}}
%EndExpansion
N_{i}$ and $r\notin(J(R)M:M)$. Then $rm\in N_{i}$ for all $i=1,...,k$ which
gives $m\in M$-$rad(N_{i})$ for all $i=1,...,k.$ Since $%
%TCIMACRO{\dbigcap \limits_{i=1}^{k}}%
%BeginExpansion
{\displaystyle\bigcap\limits_{i=1}^{k}}
%EndExpansion
M$-$rad(Ni)=M$-$rad\left(
%TCIMACRO{\dbigcap \limits_{i=1}^{k}}%
%BeginExpansion
{\displaystyle\bigcap\limits_{i=1}^{k}}
%EndExpansion
N_{i}\right)  $ \cite[Theorem 15 (3)]{Ali}, we conclude that $%
%TCIMACRO{\dbigcap \limits_{i=1}^{k}}%
%BeginExpansion
{\displaystyle\bigcap\limits_{i=1}^{k}}
%EndExpansion
N_{i}$ is a quasi $J$-submodule. By using the similar argument and the
equality $\prod\limits_{i=1}^{k}M$-$rad(Ni)=M$-$rad\left(
%TCIMACRO{\dbigcap \limits_{i=1}^{k}}%
%BeginExpansion
{\displaystyle\bigcap\limits_{i=1}^{k}}
%EndExpansion
N_{i}\right)  $ \cite[Proposition 2.14.]{Ece}, $\prod\limits_{i=1}^{k}N_{i}$
is also a quasi $J$-submodules of $M$.
\end{proof}

The converse of the above proposition can be achieved under certain conditions.

\begin{proposition}
Let $M$ be a finitely generated faithful multiplication $R$-module and
$N_{1},N_{2},...,N_{k}$ be quasi primary submodules of $M$ such that
$\sqrt{N_{i}:M}$ are not comparable for all $i=1,...,k.$ If $%
%TCIMACRO{\dbigcap \limits_{i=1}^{k}}%
%BeginExpansion
{\displaystyle\bigcap\limits_{i=1}^{k}}
%EndExpansion
N_{i}$ or $\prod\limits_{i=1}^{k}N_{i}$ is a quasi $J$-submodule of $M$, then
$N_{i}$ is a quasi $J$-submodule of $M$ for each $i=1,...,k$.
\end{proposition}

\begin{proof}
Suppose that $N_{i}$ $(i=1,...,k)$ is a quasi primary submodule of $M.$ Then
$(N_{i}:M)$ is a quasi primary ideal of $R$ for all $i=1,...,k$, \cite[Lemma
2.12]{Hosein}. If $%
%TCIMACRO{\dbigcap \limits_{i=1}^{k}}%
%BeginExpansion
{\displaystyle\bigcap\limits_{i=1}^{k}}
%EndExpansion
N_{i}$ is a quasi $J$-submodule, we conclude from Theorem \ref{1} that
$\left(
%TCIMACRO{\dbigcap \limits_{i=1}^{k}}%
%BeginExpansion
{\displaystyle\bigcap\limits_{i=1}^{k}}
%EndExpansion
N_{i}:M\right)  =%
%TCIMACRO{\dbigcap \limits_{i=1}^{k}}%
%BeginExpansion
{\displaystyle\bigcap\limits_{i=1}^{k}}
%EndExpansion
\left(  N_{i}:M\right)  $ is a quasi $J$-ideal of $R$. Hence, $\left(
N_{i}:M\right)  $ is a quasi $J$-ideal of $R$ for all $i=1,...,k$ by
\cite[Proposition 5]{Haniece}. Thus, $N_{i}$ is a quasi $J$-submodule of $M$
for all $i=1,...,k$ by Theorem \ref{1}. The proof of the finite product case
is similar since $\left(  \prod\limits_{i=1}^{k}N_{i}:M\right)  =\prod
\limits_{i=1}^{k}\left(  N_{i}:M\right)  $ and by using Theorem \ref{1} and
\cite[Proposition 6]{Haniece}.
\end{proof}

Let $R$ be a ring and $M$ be an $R$-module. The idealization ring of $M$ is
the set $R(+)M=R\oplus M=\left\{  (r,m):r\in R\text{, }m_{2}\in M\right\}  $
with coordinate-wise addition and multiplication defined as $(r_{1}%
,m_{1})(r_{2}m_{2})=(r_{1}r_{2},r_{1}m_{2}+r_{2}m_{1})$. If $I$ is an ideal of
$R$ and $N$ a submodule of $M$, then $I(+)N$ is an ideal of $R(+)M$ if and
only if $IM\subseteq N$. It is well known that if $I(+)N$ is an ideal of
$R(+)M$, then $\sqrt{I(+)N}=\sqrt{I}(+)M$. Moreover, we have
$J(R(+)M)=J(R)(+)M$, \cite{Anderson4}. Next, we characterize quasi $J$-ideals
in any idealization ring $R(+)M$.

\begin{theorem}
\label{6}Let $I$ be an ideal of of a ring $R$ and $N$ be a submodule of an
$R$-module $M$. Then $I(+)N$ is a quasi $J$-ideal of $R(+)M$ if and only if
$I$ is a quasi $J$-ideal of $R$.
\end{theorem}

\begin{proof}
Suppose $I(+)N$ is a quasi $J$-ideal of $R(+)M$ and let $a,b\in R$ such that
$ab\in I$ and $a\notin J(R)$. Then $(a,0)(b,0)\in I(+)N$ and $(a,0)\notin
J(R(+)M)$. Therefore, $(b,0)\in\sqrt{I(+)N}=\sqrt{I}(+)M$ and so $b\in\sqrt
{I}$ as needed. Conversely, suppose $I$ is a quasi $J$-ideal of $R$. Let
$\left(  r_{1},m_{1}\right)  ,\left(  r_{2},m_{2}\right)  \in$ $R(+)M$ such
that $\left(  r_{1},m_{1}\right)  \left(  r_{2},m_{2}\right)  =(r_{1}%
r_{2},r_{1}m_{2}+r_{2}m_{1})\in I(+)N$ and $\left(  r_{1},m_{1}\right)  \notin
J(R(+)M)=J(R)(+)M$. Then $r_{1}r_{2}\in I$ and $r_{1}\notin J(R)$ which imply
that $r_{2}\in\sqrt{I}$. Thus, $\left(  r_{2},m_{2}\right)  \in\sqrt
{I}(+)M=\sqrt{I(+)N}$ and $I(+)N$ is a quasi $J$-ideal of $R(+)M$.
\end{proof}

In view of Theorem \ref{6}, we have

\begin{corollary}
Let $I$ be an ideal of of a ring $R$ and $M$ be a finitely generated faithful
multiplication $R$-module. If $IM$ is a quasi $J$-submodule of $M$, then
$I(+)N$ is a quasi $J$-ideal of $R(+)M$ for any submodule $N$ of $M$.
\end{corollary}

\begin{proof}
The result follows by Theorem \ref{6} and (1) of Theorem \ref{1}.
\end{proof}

We note that if $I(+)N$ is a quasi $J$-ideal of $R(+)M$, then $N$ need not be
a quasi $J$-submodule of $M$. For example, while $0(+)\overline{0}$ is a quasi
$J$-ideal of $%
%TCIMACRO{\U{2124} }%
%BeginExpansion
\mathbb{Z}
%EndExpansion
(+)%
%TCIMACRO{\U{2124} }%
%BeginExpansion
\mathbb{Z}
%EndExpansion
_{6}$ by Theorem \ref{6}, but $\overline{0}$ is not quasi $J$-submodule of $%
%TCIMACRO{\U{2124} }%
%BeginExpansion
\mathbb{Z}
%EndExpansion
_{6}$. For example, $2.\bar{3}=\bar{0}$ but $2\notin(J(%
%TCIMACRO{\U{2124} }%
%BeginExpansion
\mathbb{Z}
%EndExpansion
)%
%TCIMACRO{\U{2124} }%
%BeginExpansion
\mathbb{Z}
%EndExpansion
_{6}:%
%TCIMACRO{\U{2124} }%
%BeginExpansion
\mathbb{Z}
%EndExpansion
_{6})=\left\langle 6\right\rangle $ and $\bar{3}\notin M$-$rad(\bar{0}%
)=\bar{0}$.


\begin{thebibliography}{99}                                                                                               %


\bibitem {Majed}M. M. Ali, Residual submodules of multiplication modules,
Beitr Algebra Geom, 46 (2) (2005), 405--422.

\bibitem {Ali}M. M. Ali, Idempotent and nilpotent submodules of multiplication
modules, Comm. Algebra, 36 (2008), 4620-4642.

\bibitem {Anderson1}D. D. Anderson, Multiplication ideals, multiplication
rings and the ring R(x), Can. J. Math., 28 (4) (1976), 760-768.

\bibitem {Anderson3}D.D. Anderson and S. Valdes-Leon, Factorization in
commutative rings with zero divisors II, Factorization in integral domains,
Lecture Notes in Pure and Appl. Math., 189, (1997), 197-219.

\bibitem {Anderson2}D. D. Anderson, M. Axtell, S.J. Forman, J. Stickles, When
are associates unit multiples?, Rocky Mountain J. Math., 34 (2004), 811-828.

\bibitem {Anderson4}D. D. Anderson, M. Winders, Idealization of a Module,
Journal of Commutative Algebra, 1 (1) (2009), 3-56.

\bibitem {Bernard}A. Barnard, Multiplication modules. J. Algebra 71 (1981), 174--178.

\bibitem {Bast}Z. A. El-Bast and P. F. Smith, Multiplication modules, Comm. in
Algebra, 16 (1988), 755-779.

\bibitem {Hosein}F. M. Hosein, S. Mohdi, Quasi-primary Submodules Satisfying
the Primeful Property I, Hacet. J. Math. Stat., 45 (5) (2016), 1421-1434.

\bibitem {Hani}H. A. Khashan, A. B. Bani-Ata, $J$-ideals of commutative rings,
International Electronic Journal of Algebra, 29 (2021), 148-164.

\bibitem {Haniece}H. A. Khashan, E. Yetkin Celikel, Quasi $J$-ideals of
commutative rings, (submitted).

\bibitem {Suat}S. Koc, U. Tekir, $r$-Submodules and $sr$-Submodules, Turk. J.
Math. 42 (2018), 1863--1876.

\bibitem {Lu}C. P. Lu, M-radicals of submodules in modules, Math. Japonica, 34
(2) (1989), 211-219.

\bibitem {Moh}R. Mohamadian, r-ideals in commutative rings, Turkish Journal of
Mathematics, 39 (2015), 733-749.

\bibitem {Ece}H. Mostafanasab, E. Yetkin Celikel, U. Tekir, A. Y. Darani, On
2-absorbing primary submodules over commutative rings, An. St. Ovidius
Constanta, Seria Matematica 24 (1) (2016), 335-351.

\bibitem {Rib}P. Ribenboim, Algebraic Numbers. Wiley, 1972.

\bibitem {Smith}P. Smith, Some remarks on multiplication modules, Arch. Math.,
50 (1988), 223-235.

\bibitem {Tekir}U. Tekir, S. Koc, K. H. Oral, $n$-ideals of Commutative Rings,
Filomat, 31 (10) (2017), 2933-2941.
\end{thebibliography}
\end{document}